\documentclass[12pt]{article}
\usepackage{amssymb,amsmath, amsthm, amscd,ifthen}
\usepackage[T2A]{fontenc}
\usepackage[cp1251]{inputenc}
\usepackage[russian,english]{babel}
\usepackage[dvips]{graphicx}
\usepackage{indentfirst}

\newcommand{\R}{{\mathbb R}}
\newcommand{\N}{{\mathbb N}}

\newtheorem*{gypo}{Conjecture}
\newtheorem{opr}{Definition}
\newtheorem{cor}{Corollary}
\newtheorem{thm}{Theorem}
\newtheorem{lem}{Lemma}
\DeclareMathOperator{\conv}{conv}
\title{Diameter graphs in $\R^4$ }
\begin{document}
\author{Andrey Kupavskii\footnote{Moscow Institute of Physics and Technology} }

\maketitle

\begin{abstract} A \textit{diameter graph in $\R^d$ } is a graph, whose set of vertices is a finite subset of $\R^d$ and whose set of edges is formed by pairs of vertices that are at diameter apart. This paper is devoted to the study of different extremal properties of diameter graphs in $\R^4$ and on a three-dimensional sphere. We prove an analogue of V\'azsonyi's and Borsuk's conjecture for diameter graphs on a three-dimensional sphere with radius greater than $1/\sqrt 2$. We prove Schur's conjecture for diameter graphs in $\R^4.$ We also establish the maximum number of triangles a diameter graph in $\R^4$ can have, showing that the extremum is attained only on specific Lenz configurations. \end{abstract}

\section{Introduction}

The following question was raised by Borsuk in 1933 \cite{Bor}: is it true that any set of diameter~1 in $\R^d$ can be partitioned into $d+1$ parts of strictly smaller diameter? The positive answer to this question is called Borsuk's conjecture. Borsuk gave a positive answer to this question for $d=2,$ and later the same was proved for $d=3$ (see \cite{Rai1, Rai3}). Borsuk's conjecture was disproved by Kahn and Kalai in 1993 \cite{KK}. In that paper they constructed a \textit{finite} set of points in dimension 2016 such that it cannot be partitioned into 2017 parts of smaller diameter.  The bounds on the minimum dimension of the counterexample were obtained by several authors. Very recently, Bondarenko \cite{Bond} disproved Borsuk's conjecture in dimensions $d\ge 65$.

An analogue of Borsuk's conjecture for finite sets is well-studied. A natural notion to work with in the finite case is that of \textit{diameter graph}.
A \textit{diameter graph in $\R^d$  } is a graph, whose set of vertices is a finite subset of $\R^d$ and whose set of edges is formed by the pairs of vertices that are at diameter apart. Next we work only with  sets of diameter 1. For a finite set $X$ of unit diameter denote by $G(X)$ the diameter graph with the vertex set $X$. In terms of diameter graphs, Borsuk's problem for finite sets can be formulated as follows: is it true that for any $X\subset \R^d$ we have $\chi(G(X))\le d+1$? Here $\chi(G)$ is the chromatic number of the graph.

In \cite{PH} Hopf and Pannwitz proved that the number of edges in any diameter graph in $\R^2$ is at most $n$, which easily implies Borsuk's conjecture for finite sets on the plane. V\'azsonyi conjectured, that any diameter graph in $\R^3$ on $n$ vertices can have at most $2n-2$ edges. It is easy to see that Borsuk's conjecture for finite sets in $\R^3$ follows from this statement. This conjecture was proved independently by Gr\"unbaum \cite{GR}, Heppes \cite{Hep2} and Straszewicz \cite{St}.

In this paper we prove Borsuk's and V\'azsonyi's conjecture for finite sets on a three-dimensional sphere $S^3_r$ of radius $r>1/\sqrt 2$ (note that we consider sets of Euclidean diameter~1). It is easy to see that V\'azsonyi's conjecture fails for $S^3_{1/\sqrt 2}$ and that Borsuk's conjecture fails for $S^3_{\sqrt{2/5}}$. Diameter graphs on $S^3_r$ are discussed in Section \ref{sec2}.

As we already discussed, the study of the maximum number of edges in a diameter graph is related to Borsuk's conjecture. Surely, it has an independent interest. Extremal properties of diameter graphs and unit distance graphs were extensively studied. A \textit{Unit distance graph} in $\R^d$ is a graph, whose set of vertices is a finite subset of $\R^d$ and whose set of edges is formed by pairs of vertices that are at unit distance apart (here we do not demand that the set of vertices is of diameter 1).

Denote by $D_d(l,n)$ ($U_d(l,n)$) the maximum number of cliques of size $l$ in a diameter (unit distance) graph on $n$ vertices in $\R^d$. Erd\H os \cite{Erd,Erd2} studied $U_d(2,n)$ and $D_d(2,n)$ for different $d$. He showed that for $d\ge 4$ we have $U_d(2,n), D_d(2,n) = \frac{\lfloor d/2\rfloor-1}{2\lfloor d/2\rfloor}n^2+\bar o(n^2).$ Brass \cite{Br} and van Wamelen \cite{Wam} determined $U_4(2,n)$
 for all $n$. Swanepoel \cite{Swan} determined $U_d(2,n)$ for even $d\ge 6$ and sufficiently large $n$ and determined
 $D_d(2,n)$ for $d\ge 4$ and sufficiently large $n$. He also proved some results concerning the stability of the extremal configurations. We refine the result of Swanepoel concerning $D_4(2,n)$ by giving a reasonable bound on $n$: we show that his result holds for $n\ge 52$.

 Functions $D_d(l,n),$ $U_d(l,n)$ and similar functions were studied in several papers. In particular, the following conjecture was raised in \cite{Sch}:

 \begin{gypo}[Schur et. al., \cite{Sch}]
 We have $D_d(d,n) = n$ for $n\ge d+1$.
 \end{gypo}

This was proved by Hopf and Pannwitz for $d=2$ in \cite{PH} and for $d=3$ by Schur et. al. in \cite{Sch}. They also proved that $D_d(d+1,n)=1$. In \cite{Philip2} the authors proved that Schur's conjecture holds in some special case:

\begin{thm}[Mori\'c, Pach, \cite{Philip2}] \label{thPM} The number of $d$-cliques in a graph of diameters on $n$ vertices in $\R^d$ is at most $n$, provided that any two $d$-cliques share at least $d-2$ vertices.
\end{thm}

In this paper we prove Schur's conjecture for $d=4$. Moreover, we determine the exact value of $D_4(3,n)$ for large $n$. This completes the full description of functions $D_4(l,n)$ for large $n$. We also improve the result from Theorem \ref{thPM} (in Section \ref{sec3}).

In the next section we discuss diameter graphs on three-dimensional spheres, and in Section~\ref{sec3} we discuss diameter graphs in $\R^4$.

\section{Diameter graphs on the three-dimensional sphere}\label{sec2}
In this section we prove the following theorem:

\begin{thm}\label{th1} Let $X$ be a finite subset of diameter 1 on $S^3_r$, $|X|=n$. If $r> 1/\sqrt 2$, then:
\begin{enumerate}
\item  $G = G(X)$ has at most $2n-2$ edges.
\item  $\chi(G)\le 4$.
\item Any two odd cycles in $G$ have a common vertex.
\end{enumerate}
\end{thm}
The proof is based on the approach which was suggested by V. Dol'nikov \cite{Dol} and developed by K. Swanepoel \cite{Swan1}. The author is grateful to A.V. Akopyan, who suggested the key idea of reduction to the great sphere $S$ (see the proof of the theorem). A.V. Akopyan proved Borsuk's and V\'azsonyi's conjecture on the sphere before the author (private communication) but he has not written the proof. Moreover, he claims that the proof works also for the three-dimensional hyperbolic space.

We will need the following lemmas.

\begin{lem}\label{lem1}
Fix some natural $d\ge 2$. Let $X$ be a subset of $S^{d-1}_r$ of unit diameter. If $r>\sqrt{d/(2d+2)}$, then $X$ lies in an open hemisphere of $S^{d-1}_r.$
\end{lem}
\begin{proof}
Since $X$ is a subset in $\R^d$, by Jung's theorem, $X$ can be covered by a ball $B$ of radius $\sqrt{d/(2d+2)}$. The sphere that bounds the ball $B$ and $S^{d-1}_r$ intersect in a sphere $S$ of radius not greater than $\sqrt{d/(2d+2)}$, and the intersection of $B$ and $S^{d-1}_r$ lies entirely in the open hemisphere bounded by the great sphere $S'\subset S^{d-1}_r,$ which is parallel to $S$.
\end{proof}

The next lemma is a modification of Lemma 3 from \cite{Swan1}.
\begin{lem}\label{lem2} Fix some natural $d\ge 2$. Let $x_1,\ldots, x_k$ and $\sum_{i=1}^k \lambda_i x_i$ be distinct vectors of length $a>0$ in $\R^d$, where $\lambda_i\ge 0$. Fix some $b>0$. Suppose that for some vector $y\in \R^d$ we have $\|y-x_i\|\le b$ for each $i=1,\ldots, k$. Then $\|y - \sum_{i=1}^k \lambda_i x_i\| <b,$ if $\|y\|^2+a^2-b^2>0$.
\end{lem}
\begin{proof}
Since none of the $x_i$ are collinear, from the strict convexity of the Euclidean norm we get: $$1= \frac {\|\sum_{i=1}^k \lambda_ix_i\|}{a}<\sum_{i=1}^k \lambda_i.$$
For each $i$ we have $$b^2\ge\langle y-x_i, y-x_i\rangle = a^2+ \|y\|^2 - 2 \langle y,x_i\rangle,$$ and we obtain
$\|y\|^2 +a^2-b^2\le 2 \langle y,x_i\rangle.$

Thus we have

\begin{multline*}\|y-\sum_{i=1}^k \lambda_i x_i\|^2 = \|y\|^2 - 2 \sum_{i=1}^k \lambda_i\langle y,x_i\rangle + a^2 \le \|y\|^2- \\- (\|y\|^2+a^2-b^2)\sum_{i = 1}^k \lambda_i +a^2  =(\|y\|^2+a^2-b^2)(1-\sum_{i = 1}^k \lambda_i)+b^2 <b^2,\end{multline*}

since  $\|y\|^2+a^2-b^2>0$.
\end{proof}

\begin{opr} The spherical convex hull $\conv_S(x_1,\ldots,x_k)$ of the points $x_1,\ldots, x_k$ that lie in a hemisphere on the sphere $S'$ centered at the point $O$ is the intersection of the sphere $S'$ and the cone, formed by the vectors $Ox_i$ (the cone consists of all vectors of the form $\sum_{i=1}^k \lambda_i Ox_i, \lambda_i\ge 0$). The vertices of $\conv_S(x_1,\ldots,x_k)$ are the points of $\conv_S(x_1,\ldots,x_k)$ that correspond to vectors that cannot be expressed as a non-trivial convex combination of the other vectors forming the cone. Alternatively, these are such points $y_1,\ldots, y_l$ of $\conv_S(x_1,\ldots,x_k)$ that $\conv_S(y_1,\ldots, y_l)=\conv_S(x_1,\ldots,x_k)$ and the set $\{y_1,\ldots, y_l\}$ is minimal.
\end{opr}

It is fairly easy to show that the set of vertices of $\conv_S(x_1,\ldots,v_k)$ is a subset of $\{x_1,\ldots,x_k\}$.
For two points $x_1,x_2$ on the sphere $S$ we denote by  $\overset{\frown}{x_1x_2}$  the shorter arc of the great circle that contains these two points. By $\|x_1-x_2\|_S$ we denote the length of the arc. For the points $x_1,\ldots,x_k$ on the sphere $S$ we denote by $S(x_1,\ldots, x_k)\subseteq S$ the great sphere of minimal dimension that contains $x_1,\ldots, x_k$.

\begin{lem}\label{lem3}
Let $X$ be a subset of $S^2_r$ of diameter 1. If $r>\sqrt{3/8}$ then for any $a_1,b_1,a_2,b_2$ such that $(a_i,b_i) \in E(G(X)), i=1,2,$ the arcs $\overset{\frown}{a_1b_1}$ and $\overset{\frown}{a_2b_2}$ intersect.
\end{lem}

\begin{proof} By Lemma \ref{lem1}, $X$ lies in an open hemisphere of $S^2_r$. Suppose that the arcs do not intersect. Consider the spherical convex hull of the points $a_1,a_2,b_1,b_2$. We have the following two possibilities.

First, the spherical convex hull is a spherical triangle. Without loss of generality, assume that the vertices of the triangle are $a_1,b_1,b_2$. Then we can apply Lemma \ref{lem2} for the points $a_1,b_1,b_2$ as $x_i$, $a_2$ as $\sum_{i=1}^3 \lambda_ix_i$ and $b_2$ as $y$. We put $a=b=1$ and obtain that, on the one hand, $\|a_2-b_2\|$ should be strictly less than one, but on the other, these two vertices are connected by an edge, a contradiction.

Second, the convex hull is a spherical quadrilateral with $\overset{\frown}{a_1b_1}$ and $\overset{\frown}{a_2b_2}$ as two edges. Suppose that the other two edges of the quadrilateral are $\overset{\frown}{a_1a_2}$ and $\overset{\frown}{b_1b_2}$, so $\overset{\frown}{a_1b_2}$ and $\overset{\frown}{a_2b_1}$ are diagonals, and that they intersect at a point $x$.  By the triangle inequality for the sphere we obtain that $\|a_1-x\|_S+\|x-b_1\|_S> \|a_1-b_1\|_S,$ $\|a_2-x\|_S+\|x-b_2\|_S> \|a_2-b_2\|_S.$ Consequently, at least one of the following two inequalities hold: $\|a_1-b_2\|_S = \|a_1-x\|_S+\|x-b_2\|_S> \|a_1-b_1\|_S$ or $\|a_2-b_1\|_S = \|a_2-x\|_S+\|x-b_1\|_S> \|a_1-b_1\|_S$. Thus, either $\|a_1-b_2\|>1$ or $\|a_2-b_1\|>1$.
\end{proof}

\begin{proof}[Proof of Theorem \ref{th1}] Consider a set $X$ of diameter 1 on the sphere $S^3_r$ and the graph $G = G(X) = (V, E)$. By $N(v)$ we denote the set of neighbors of $v\in V$. Hereinafter $\conv_S(N(v))$ is the set on the two-dimensional sphere $S^2(v)$, which is the intersection of $S^3_r$ and the sphere of unit radius with the center $v$. The convex hull is taken with respect to $S^2(v)$.

\begin{lem}\label{lem4}  For any two points $u,v\in V$ and any two points $x\in \conv_S(N(v)),y\in \conv_S(N(u))$ we have $\|x-y\|, \|x-u\|\le 1$.  Moreover, if $x$ is not a vertex of $\conv_S(N(v))$, then $\|x-y\|,  \|x-u\|< 1.$
\end{lem}
\begin{proof} Consider an arbitrary point $z$ on the sphere $S^3_r$ such that $\|z-v\|\le 1$ and for any $w\in N(v)$ we have $\|z-w\|\le 1$. We will prove that  if $x'$ is not a vertex of $\conv_S(N(v))$, then $\|z-x'\|<1$. Inequalities in the lemma follow from this. Indeed, first one have to apply this to $u$ (as $z$) and $x$, $x\notin N(v)$, as $x'$. It is possible to do so since $u$ is at less than unit distance apart from any vertex from $V$. One obtains that $\|x-u\|< 1$ for any $u \in V$. Then one applies the statement again to $x,y$ (with $x$ being $z$, and $y$ being $x'$).

For some $w \in N(v)$ consider a vector $Ow$, where $O$ is the center of $S^2(v)$, and the hyperplane $\pi$ that is orthogonal to $Ow$ and passes through $O$. The intersection of $\pi$ and $S^3_r$ is the great sphere $S'$ that contains $v$.  The great circle $S(v,w)$ that contains $v,w$ lies in the plane which is orthogonal to $\pi$, which means that the minimum of the distance between $w$ and the points of $S'$ is attained at one of the two points of $S(v,w)\cap S'$. Since $r>1/\sqrt 2$, the point $O$ lies on the segment that connects the center of $S^3_r$ and $v,$ and thus $v$ is closer to $w$ than to the other point from $S(v,w)\cap S'$. Consequently, for any point $s\neq v$ that lies on $S'$ we have $\|s-w\|>1$, so all points of $X\backslash \{v\}$ and $z$ must lie on the side of $\pi$ that contains $w$. Otherwise $S'$ and $S^2(w)$ would intersect in at least two points, which is impossible. Therefore, $X\backslash \{v\}$ lies in the intersection of the \textit{open} hemisphere of $S^3_r$, which is bounded by $S'$ and contains $w$, and the spherical cap with the center at $v$ bounded by $S^2(v)$.

Consider the projection $z'$ of $z$ on the hyperplane that contains $S^2(v)$.
From the above considerations carried out for an arbitrary vertex of $\conv_S(N(v))$ denoted by $w$ we get that $\langle Oz',Ow\rangle>0.$ If $\|z'-w\| = b$ and $\|O-w\| = a,$ then $\|O-z'\|^2= b^2-a^2+2\langle Oz',Ow\rangle>b^2-a^2.$ Thus we can apply Lemma \ref{lem2} and obtain that $\|z'-x\|<\max_{w\in N(v)} \|z'-w\|$. Consequently, $\|z-x\|<\max_{w\in N(v)} \|z-w\|\le 1.$
\end{proof}

From Lemma \ref{lem4} we obtain that the set $X' = \bigcup_{v\in X} \Bigl( \{v\}\cup \conv_S(N(v))\Bigr)$ is a set of diameter~1.
By Lemma~\ref{lem1} $X'$ lies in an open hemisphere $H\subset S^3_r$. Denote by $S$ the diametral sphere which bounds $H$.

For a vertex $v\in V$ we denote by $w_1,\ldots w_s \in V$ the neighbors of $v$ in $G$. For $i = 1,\ldots, s$ let $u_i,u_i'$ be the points of the intersection of the sphere $S$ and the great circle $S(v,w_i)$ in $S^3_r$, where $u_i$ is closer to $w_i$ and $u_i'$ is closer to $v$. Denote by $R(v)$ the set $\conv_S(u_1,\ldots,u_s)$ on the sphere $S,$ and by $B(v)$ the set $\conv_S(u'_1,\ldots,u'_s),$ which is symmetric to $R(v)$ with respect to the center of $S$.

We note that the following important property of this ``projection'' holds. Namely, for any point $u$ in $\conv_S(u_1,\ldots,u_s)$ the arc $\overset{\frown}{vu}$ intersects $\conv_S(N(v))$ at some point $w$. We argue in terms of the vectors that correspond to the points on the sphere $S_r^3$. By abuse of notation for the vectors in this paragraph we use the same notation as for the points. Suppose the vector $u=\sum_{i=1}^k\lambda_i u_i,$ where $\lambda_i\ge 0$. Then the great circle $S(v,u)$ is formed by vectors of the form
$c_1v+c_2(\sum_{i=1}^k\lambda_i u_i)$, where $c_1,c_2\in \R$ are arbitrary, with the only condition that one of them is non-zero. Remind that the points $w_1,\ldots,w_s$ lie on the sphere $S^2(v)$ with the center at $O$. For each point $w$ in $\conv_S(N(v))$ the corresponding vector on $S_r^3$ may be expressed as a combination of vector $v$ and of vectors $Ow_i$. On the other hand, for each $i=1,\ldots, k$ vector $Ow_i$ is a combination of $v$ and $u_i$. Put $w$ to be a point on $S^2(v)$ such that the corresponding vector on $S_r^3$ is $c'v + \sum_{i=1}^k\lambda_i Ow_i$. Then, if instead of $Ow_i$ we substitute a combination of $u_i$ and $v$, we obtain a point on  $S(v,u)$. Surely, this is a point of the arc $\overset{\frown}{vu}$. The property is justified.
\begin{lem}\label{lem5}
1. For $u\neq v \in V$ the sets $R(v)$ and $R(w)$ do not intersect.

2. Suppose that for some $u,v\in V$ the sets $R(v)$ and $B(u)$ intersect. Then the intersection is a single point and in this case $(u,v)\in E$. Moreover, this point is a vertex of a spherical polyhedron $R(v)$, if $\deg u\ge 2$, and is a vertex of a spherical polyhedron $R(u)$, if $\deg v\ge 2.$
\end{lem}
\begin{proof} \textbf{1.} Suppose that the sets $R(v)$ and $R(w)$ intersect at a point  $x$. Consider the arcs $\overset{\frown}{vx}$ and $\overset{\frown}{wx}.$ Suppose they do not lie on the same diametral circle. By the property discussed befor the lemma, the arcs $\overset{\frown}{vx}$ and $\overset{\frown}{wx}$ intersect $\conv_S(N(v))$ and $\conv_S(N(w))$ at points $x_v$ and $x_w$ respectively.

Consider the great two-dimensional sphere $S(x,v,w)$ in $S_r^3$. The arcs $\overset{\frown}{vx_v}$ and $\overset{\frown}{wx_w}$ do not intersect. Applying Lemma \ref{lem3} we get that the distance between some of the points $v,w,x_v,x_w$ is greater than 1. On the other hand, all these points belong to $X'$, which is of diameter 1, a contradiction. Thus the arcs lie on the same diametral circle, and $v$ and $w$ must coincide. Indeed, if not then $v,w,x_v,x_w$ are four distinct points on one half-circle, and  either $\|v-x_w\|>1$ or $\|w-x_v\|>1$.

\textbf{2.} Suppose that the sets $R(v)$ and $B(w)$ intersect at a point  $x$. If the arcs $\overset{\frown}{vx}$ and $\overset{\frown}{wx}$ do not lie on the same diametral circle, then we can apply the considerations from the previous part.

If these two arcs lie on the same diametral circle, then $v\in N(w)$ and vice versa. Indeed, if $\|v-w\|<1,$ then  $\|x_v-x_w\|>1$, where $x_v = S(v,x)\cap \conv_S(N(v))$. On the other hand, according to Lemma~\ref{lem4}, $\|x_v-x_w\|\le 1$.

The second statement of point 2 of Lemma \ref{lem5} follows easily from the second part of Lemma~\ref{lem4}. \end{proof}

We may assume that $G$ does not have vertices of degree $\le 1$.

We construct a bipartite double cover $C=(V(C),E(C))$ of $G$, which has a symmetric drawing on $S$. We choose a point $c(v)$ in the interior of $R(v)$, and the antipodal point $c'(v)$ in the interior of $B(v)$. We connect all vertices of $R(v)$ with $c(v)$ by great arcs (since all the vertices in $G$ have degree $\ge 2$, by Lemma \ref{lem5} each neighbor of $v$ corresponds to some vertex of $R(v)$). We also draw antipodal arcs from vertices of $B(v)$ to $c'(v)$. The set $V(C)$ consists of $c(v), c'(v)$, where $v\in V$; the set of edges $E(C)$ consists of all pairs $c(v), c'(w), v,w\in V$, that are joined by curves that consist of two great arcs (one in R(v), the other in B(w)) that share a point. What we described before is thus the drawing of $C$ on $S$. It is easy to see that if for any $v\in V$ we correspond $c(v),c'(v)$ to $v$, then we indeed get a double covering of $G$. Moreover, $C$ is bipartite, since we can color $c(v),v\in V,$ in red and $c'(v), v\in V$ in blue. This is a proper coloring according to Lemma \ref{lem5}.

The graph $C$ is a planar bipartite graph on $2n$ vertices, so it has at most $4n-4$ edges. Consequently, graph $G$ has at most $2n-2$ edges and the first point of Theorem \ref{th1} is proved.

For any graph $G$ such that any subgraph $H=(V(H),E(H))$ of $G$ satisfies $|E(H)|\le 2|E(H)|-2$  it is easy to show that $\chi(G)\le 4$. Indeed, assume that $n_0$ is the minimal $n$ such that there is a graph $G$ of order $n$ satisfying the above described property and such that $\chi(G)\ge 5$. $G$ contains a vertex $v$ of degree $\le 3$. By minimality of $n_0,$ $\chi(G\backslash \{x\})\le 4.$ But then we can color $v$ in the color that differs from colors of its neighbors and obtain a proper coloring of $G$ in four colors.

To prove the last point of Theorem \ref{th1} we note that each odd cycle in $G$  corresponds in the drawing of $C$ described above to a closed self-symmetric curve on the sphere without self-intersections. Any two such curves must intersect. But they can intersect only in $c(v)$ (and $c'(v)$) for some $v\in V$. That means that the corresponding odd cycles in $G$ share vertex $v$. The proof of the theorem is complete. \end{proof}

It is worth noting that in the proof the analogous statement for diameter graphs in $\R^3$ given in the paper \cite{Swan1} there is a slight inaccuracy related to the intersections of sets $R(x),B(y)$. In \cite{Swan1} Swanepoel used the following lemma, which is an analogue of point 2 of Lemma \ref{lem5}:
\begin{lem}[Lemma 2 from \cite{Swan1}] If $R(x)$ and $B(y)$ intersect, then $xy$ is a diameter and $R(x)\cap B(y) = \{y-x\}$.\end{lem}
Then  Swanepoel constructed a bipartite double cover using the same considerations as above. However, this lemma is not enough to construct a bipartite double cover which is a planar graph, so the final conclusion from \cite{Swan1}, ``By Lemmas 1 and 2 no edges cross, and the theorem follows,'' is wrong. The important thing missing is that, after deleting all the vertices of degree 1, each point  in $R(x), B(x)$ that correspond to diameters in the graph must be a vertex of the spherical polygon $R(x),B(x)$. The problem is that Lemma 2 does not exclude the following configuration: $R(x)$ and $B(y)$ are arcs $\overset{\frown}{u_xv_x}$ and $\overset{\frown}{u_yv_y}$ that intersect at the interior point $z = y-x.$ This correspond to the situation when $x$ is connected by an edge to $y, x+u_x,x+v_x$ (see \cite{Swan1}), and $y$ is connected to $x, u_y, v_y$. The conditions of Lemma 2 from \cite{Swan1} are satisfied in this situation, but if one tries to construct a drawing of $C$ as described above, he ends up with a drawing that has self-intersections.

Fortunately, this configuration is impossible to get in $\R^3$, since the statement, analogous to the second part of the point 2 of Lemma \ref{lem5} holds for diameter graphs in $\R^3$ (and it is in fact easy to deduce from Lemma 3 from \cite{Swan1}).

Nevertheless, if we consider the sphere $S^3_r$ with $r=1/\sqrt 2$, then we indeed can get the configuration described above, if we try to carry out the proof of Theorem \ref{th1} in this case. The graph $G$ we need to consider is a complete bipartite graph on $2n$ vertices with equal part sizes. It has a standard realization on $S^3_r$, with two parts placed on two orthogonal diametral circles. The statement of the theorem indeed does not hold for such a graph since it has $n^2$ edges. Besides, this example show that the bound on $r$ in Theorem \ref{th1} is sharp.

\section{Diameter graphs in $\R^4$}\label{sec3}

As we already mentioned,
Brass \cite{Br} and Van Wamelen \cite{Wam} determined $U_4(2,n)$:
\begin{thm}\label{thbr} For $n\ge 5,$
   $$U_4(2,n)= \begin{cases}\lfloor n^2/4\rfloor +n,\ \ \ \ \ \ \text{ if $n$ is divisible by 8 or 10}, \\
                    \lfloor n^2/4\rfloor +n-1, \ \text{ otherwise. } \end{cases}$$
\end{thm}
Thus, we have $U_4(2,n)\le n^2/4+n$ for any $n\ge 1$.
In \cite{Swan} Swanepoel established the maximum number of edges in a diameter graph in $\R^4,$ if $n$ is sufficiently large:
\begin{thm}\label{thSwan} For all sufficiently large $n$,
$D_4(2,n) = F_2(n)$, where $$F_2(n) = \begin{cases}t_2(n)+\lceil n/2\rceil +1, \text{ if } n\not\equiv 3\ \mathrm{mod}\ 4, \\
                    t_2(n)+\lceil n/2\rceil,\ \ \ \ \ \text{ if } n\equiv 3\ \mathrm{mod}\ 4, \end{cases}$$
where $t_2(n) = \lfloor n/2\rfloor\lceil n/2\rceil$ is the number of edges in a complete bipartite graph on $n$ vertices with almost equal part sizes.
\end{thm}
In this section we prove the following theorem:

\begin{thm}\label{th2}

\begin{enumerate}
\item  The statement of Theorem \ref{thSwan} holds for $n\ge 52$.
\item  For all sufficiently large $n$ we have
$D_4(3,n) = F_3(n)$, where $$F_3(n)= \begin{cases}(n-1)^2/4+ n, \ \ \ \ \ \text{ if } n\equiv 1\ \mathrm{mod}\ 4, \\
                    (n-1)^2/4+n-1, \text{ if } n\equiv 3\ \mathrm{mod}\ 4, \\
                    n(n-2)/4+ n, \ \ \ \ \ \! \text{ if } n\equiv 0\ \mathrm{mod}\ 2. \end{cases}$$
\item(Schur's conjecture in $\R^4$)  For all $n\ge 5$ we have $D_4(4,n) = n$.
\end{enumerate}
\end{thm}

\textbf{Remark.\ } It seems hard to derive any reasonable bound on $n$ from the proof of Theorem \ref{thSwan} by Swanepoel. It is due to the fact that the proof relies on the stability theorem due to Erd\H os and Simonovits (\cite{Bol}, Theorem 4.2, Section 6).\\

To prove the third part of Theorem \ref{th2} we will need the following theorem, which is derived easily from Theorems \ref{thPM} and \ref{th1}:

\begin{thm}\label{th3} Two $d$-cliques in a diameter graph $G$ in $\R^d$ cannot share exactly $d-3$ vertices. In particular, if any two $d$-cliques in $G$ share at least $d-3$ vertices, then the number of $d$-cliques in $G$ is at most the number of vertices of $G$.
\end{thm}

\begin{proof} Consider two $d$-cliques $K_1, K_2$ in $G$ that share $d-3$ vertices $v_1,\ldots, v_{d-3}$. The vertices $w_1,w_2,w_3\in K_1$ and $u_1,u_2,u_3 \in K_2$ that are different from $v_1,\ldots,v_{d-3}$ lie on a 3-dimensional sphere $S_r^3$ of radius $r = \sqrt{1-\frac{d-4}{2d-6}}>1/\sqrt 2$. Thus, we can apply part 3 of Theorem \ref{th1} to the points of $G$ that lie on $S_r^3$ and obtain that any two triangles on $S_r^3$ must share a vertex. So, some of the vertices of the triangles $u_1u_2u_3, w_1w_2w_3$ must coincide. We obtain that $K_1,K_2$ must share at least $d-2$ vertices. To finish the proof we apply Theorem \ref{thPM}.
\end{proof}

In Subsections \ref{sec31}, \ref{sec32}, \ref{sec33} we prove the first, the second  and the third part of Theorem \ref{th2} respectively.
\subsection{Number of edges}\label{sec31}
The configuration that gives the lower bound in Theorem \ref{thSwan} is called a \textit{Lenz configuration} (see \cite{Swan}).
Consider two circles $C_1$ and $C_2$ with a common center of radius $r_1$ and $r_2,$ respectively. Suppose that the circles lie in two orthogonal planes and that $r_1^2+r_2^2 =1$. A finite set $S$ is a \textit{Lenz configuration}, if  $S\subset C_1\cup C_2$ for some circles $C_1,C_2$ that satisfy the above described conditions.

Note that if a diameter graph in $\R^4$ contains a complete bipartite graph with at least three vertices in each part as a spanning subgraph, then its vertices form a Lenz configuration.

 Thus, we need to prove only the upper bound. As in \cite{Swan}, we prove that, indeed, the maximum  number of edges is attained only on the  Lenz configurations.

We will need the lemma which is a version of the famous K\H ov\'ari-S\'os-Tur\'an theorem \cite{KST}:
\begin{lem}\label{lemkst} Let $s,n \in \N, 0<c<1/2.$ If $G=(V,E)$ is a graph on $n$ vertices, $e=|E|\ge cn^2$, and if $2cn(2cn-1)(2cn-2)>(s-1)(n-1)(n-2)$, then $G$ contains a copy of $K_{s,3}$ as a subgraph.\end{lem}
\begin{proof}
Suppose $V = \{v_1,\ldots, v_n\}$ and $d_i$ is the degree of $v_i$. If
\begin{equation}\label{kst} \sum_{i=1}^n{d_i\choose 3}> (s-1){n\choose 3},\end{equation}
then, by the pigeonhole principle, some $s$ vertices from $V$ have three common neighbors. These $s$ vertices together with their three common neighbors form a copy of $K_{s,3}.$ Applying Jensen's inequality, one can check that the left-hand side is minimized when all $d_i$ are equal, so (\ref{kst}) follows from the inequality:
\[
\begin{aligned}&n\frac {2e}n\left(\frac {2e}n-1\right)\left(\frac {2e}n-2\right)>(s-1)n(n-1)(n-2) \ \Leftrightarrow\\
& 2cn(2cn-1)(2cn-2)>(s-1)(n-1)(n-2).\end{aligned}
\]\end{proof}

From Lemma \ref{lemkst} we obtain the following corollary:
\begin{cor}\label{corhr}
If $G=(V,E)$ is a graph on $n\ge 52$ vertices, $e=|E|\ge n^2/4$, then $G$ contains a copy of $K_{7,3}$ as a subgraph.
\end{cor}

Let $G$ be a graph of diameters in $\R^4$ on $n$ vertices with $D_4(2,n)\ge F_2(n)$ edges. Since $n\ge 52$, from Corollary \ref{corhr} we obtain that $G$ contains a copy of $K_{7,3}$. Suppose the set $V_1\subset V$ is a maximal subset such that $G[V_1]$ contains $K_{l,m}, l\ge 7 ,m\ge 3,$ as a spanning subgraph.

The number of edges between $V_1$ and $V\backslash V_1$ is at most $4 (|V|-|V_1|)$. Indeed, if some vertex $v$ from $V\backslash V_1$ is connected to five vertices in $V_1$, then it is connected to at least three vertices from one part of $K_{l,m}$ and it must be cocircular with the vertices of the other part. Thus we can add $v$ to the bipartite graph and obtain a contradiction with the maximality of $V_1$.

Denote $x=|V_1|\ge 10.$ We obtain the following inequality on $D_4(2,n)$:

$$D_4(2,n)\le F_2(x)+4(n-x)+|E(G[V\backslash V_1])|\le F_2(x)+4(n-x)+ (n-x)^2/4+(n-x),$$
where the last inequality follows from the fact that $D_4(2,n)\le U_4(2,n)\le n^2/4+n.$ We use that $n^2/4+n/2\le F_2(n)\le n^2/4+n/2+5/4$:
$$\begin{aligned} D_4(2,n)\le &x^2/4+x/2+5/4+5(n-x)+ (n-x)^2/4 = n^2/4+n/2-\\- &x(n-x)/2
+5/4+9(n-x)/2\le F_2(n)-(x-9)(n-x)/2+5/4.\end{aligned}$$
Thus, if $n-x\ge 3,$ then by the inequality above the graph $G$ cannot have the maximum number of edges. If $n-x=1$ or $2$, then we can use the improved bound $|E(G[V\backslash V_1])|\le (n-x)^2/4+(n-x)-5/4$ and obtain that $G$ cannot have the maximum number of edges in this case either. Thus, $n-x=0$ and the vertices of the graph $G$ form a Lenz configuration. The first part of Theorem \ref{th2} is proved.
\subsection{Number of triangles}\label{sec32}
First we show that there is a Lenz configuration on $n$ vertices with $F_3(n)$ triangles and that it is indeed the maximum number of triangles among $n$-vertex Lenz configurations.
The following lemma was stated in \cite{Swan}:
\begin{lem}\label{lemsdiam}
Let  $S$ be an $n$-vertex subset of the circle, $G=(S,E)$ is the diameter graph of $S$.
\begin{enumerate}\item If the radius of the circle $>1/\sqrt 3$, then we have $|E|\le 1$.
\item $|E|\le \begin{cases}n,\ \ \ \ \ \ \text{ if $n$ is odd}, \\
                    n-1, \ \text{if $n$ is even. } \end{cases}$\end{enumerate}
\end{lem}

Consider a Lenz configuration $V, |V|=n\ge 5$, that lies on two orthogonal circles $C_1$ and $C_2$, where $C_2$ has radius $\ge 1/\sqrt 2$. Put $V_1 = V\cap C_1, |V_1|=a$, $V_2 = V\cap C_2, |V_2|=n-a$. The number of diameters in $V_2$ is at most one, while the number of diameters in $V_1$ is at most $2\lfloor(a-1)/2\rfloor +1$. Thus the number of triangles in $G(V)$ is at most $$a+(n-a)\bigl(2\lfloor(a-1)/2\rfloor+1\bigr)=n+2(n-a)\lfloor(a-1)/2\rfloor,$$
and for each $n-2\ge a\ge 2,n\ge 5$ there is a Lenz configuration with that exact number of triangles.
It is not difficult to show that the maximum over $a$ of the number of triangles is exactly $F_3(n).$

Next we prove the following auxiliary statement concerning the number of triangles in a diameter graph:

\begin{lem}\label{lemtrbor} Any diameter graph $G=(V,E)$ in $\R^4$ on $n$ vertices has at most $4|E|/3 - 2n/3$ triangles. In particular, this quantity is at most $n^2/3+2n/3$.
\end{lem}
\begin{proof} Suppose $V=\{v_1,\ldots, v_n\},$ and $v_i$ has degree $d_i$. All neighbors of $v_i$ lie on a three-dimensional unit sphere, thus, by Theorem \ref{th1}, there are at most $2d_i-2$ edges among the neighbors of $v_i$. So the vertex $v_i$ is contained in at most $2d_i-2$ triangles. This gives the first bound on the number of triangles $t(G)$ in $G:$ $t(G)\le \sum_{i=1}^n \frac{2d_i-2}3 = 4|E|/3 - 2n/3.$ As for the second bound, we know that $|E|\le n^2/4+n$ for all $n.$ One only has to ombine these two bounds.
\end{proof}

Now we go on to the proof of the second part of Theorem \ref{th2}. Consider a graph $G=(V,E)$ with at least $F_3(n)$ triangles. We will show that, if $n$ is sufficiently large, then $G$ has exactly $F_3(n)$ triangles, and $V$ forms a Lenz configuration.
By Lemma \ref{lemtrbor}, $|E|\ge 3n^2/16$. Choose $n$ large enough, so that

\begin{equation}\label{eqtrash} (\sqrt n - 8/3)(\sqrt n - 16/3)> \frac {2^4}{3^3}(\sqrt n - 1)(\sqrt n - 2).\end{equation}

This choice will be explained later.
We apply Lemma \ref{lemkst} to the graph $G$ with $s =n/32.$ Simple calculations show that, since $|E|\ge 3n^2/16$, the conditions of Lemma \ref{lemkst} are satisfied. Thus, $G$ contains a  subgraph $K_{s,3}$ on a set $V'$. Next, as in the proof of the previous part of the theorem, we choose a maximal set $V_1\supset V'$, such that $V_1$ contains a copy of $K_{s_1,t_1}, s_1\ge s,t_1\ge 3,$ as a spanning subgraph. We run the following inductive procedure. Denote by $V(i)$ the set of available vertices at the moment $i$. At the initial moment  the set of available vertices is equal to $V$. The procedure at the step $i$ is as follows:\\
1. We choose $V_i\subset V(i-1)$ to be a maximal set in $V(i-1)$ that contains a copy of $K_{s_1,t_1}, s_1\ge s,t_1\ge 3,$ as a spanning subgraph. We require that $|V_i|\ge |V(i-1)|/32$.\\
2. We set $V(i)=V(i-1)\backslash V_i$. \\
3. If $|V(i)|\le \sqrt n$, we stop, otherwise we go on to the step $i+1$.

Note that at each step we have $E(G[V(i)])\ge 3|V(i)|^2/16$. This can be checked similar to the end of the proof of the first part of Theorem \ref{th2}. We again rely on the fact that each vertex in $V(i-1)\backslash V_i$ has at most 4 neighbors in $V_i$.

We need to prove that it is always possible to execute step 1. For that we need to verify that we can apply Lemma \ref{lemkst} with $c=3/16$. The inequality from Lemma \ref{lemkst} we need to check looks almost exactly like inequality (\ref{eqtrash}), but with $|V(i)|$ instead of $\sqrt n$. If we are to apply the step 1, then by step 3 we have $|V(i)|>\sqrt n,$ and the inequality (\ref{eqtrash}) with $|V(i)|$ instead of $\sqrt n$ also holds.

It is easy to see that procedure terminates in  $k\le 20 \ln n$ steps, since $|V(i)|\le (1-1/32)^i n=e^{\ln n-i\ln (32/31)}.$
For convenience put $V_{k+1}= V(k)$.

Now we can estimate the number $t(G)$ of triangles in $G$. Denote by $e_i$ ($t_i$) the number of edges (triangles) in $V_i$. We obtain the following estimate:

\begin{equation}\label{eqtr} t(G)\le \sum_{i=1}^{k+1}t_i +  {k\choose 2}(8(2n-2)+6n)+ 4k e_{k+1}+ (4k)^2n = \sum_{i=1}^{k}t_i + O (n\ln^2n).\end{equation}

Let us explain the inequality. The first sum counts triangles that lie entirely in one of the parts of the vertex set partition.

The second summand bounds from above the number of triangles that have one vertex in some $V_i$ and and two vertices in some $V_j$, $k\ge j>i$. First we choose $i$ and $j$. Next, the vertices of $V_i, V_j$ lie on two pairs of circles. There are at most 8 vertices of $V_i$ that lie on the circles that contain $V_j$, since otherwise we could find three vertices from $V_i$ that lie on the same circle in $V_i$ and that fall onto the same circle of $V_j$. Consequently, these two circles would coincide, and $V_i$ and $V_j$ would have to lie on the same pair of circles. This contradicts the maximality of $V_i$. The number of triangles with these 8 vertices is at most $8(2n-2)$. All vertices that do not lie on the circles that contain $V_j$ have at most four neighbors in $V_j$, thus, each is contained in at most 6 triangles. We bound the number of such triangles by $6n$.

The third term counts the number of triangles that have exactly two vertices in $V_{k+1}$. We bound their number from above as follows. First we choose an edge in $V_{k+1}$, and then for one of its endpoints we choose a neighbor from some $V_i$ (there are at most $4k$ possibilities for this choice).

The fourth summand bounds from above the number of triangles that we did not count in the first three summands. For each triangle of this type there is a part $V_i$ of the partition that contains exactly one vertex $v$ of the triangle, and two other vertices lie in the parts $V_j,V_l,\ j,l< i$. There are $n$ choices for the vertex $v$. Next, there are less than $k^2$ choices to choose two parts of the partition in which two other vertices of the triangle lie. Finally, for each $j$, each vertex from $V_i, i>j$, is connected to at most four vertices from $V_j$.

The equality in (\ref{eqtr}) is due to the following. First, $k=O(\ln n)$. Second, $|V(k)|\le \sqrt n,$ thus $e_{k+1}, t_{k+1} = O(n)$ by Lemma \ref{lemtrbor}.

Suppose $|V_1|\le n - n^{0.2}.$ One can verify that for given $a,b>0$ we have $F_3(a+b)\ge F_3(a)+F_3(b)$ for $a,b\in \N$. Besides, if $a,b\in \N, a>2b$ and $a+b$ is sufficiently large, then $F_3(a+1)+F_3(b-1)\ge F_3(a)+F_3(b)$. Therefore, we have the following bound:
$$\sum_{i=1}^{k+1}t_i\le F_3(n-n^{0.2})+F_3(n^{0.2})\le n+n^2/4 - n^{0.2}(n-n^{0.2})/2 = F_3(n)-\Omega(n^{1.2}).$$

It follows that if $|V_1|\le n - n^{0.2},$ then for sufficiently large $n$ we have $t(G)<F_3(n).$  Consider the case when $|V_1|>n-n^{0.2}.$ Remind that $V(1) = V\backslash V_1$. We have $|V(1)|<n^{0.2}$, and for a given vertex $v$ in $V(1)$ the degree of $v$ in $G[V(1)]$ is at most $n^{0.2}$. The vertex $v$ is connected to at most four vertices from $V_1$. Thus $\deg v = O(n^{0.2})$, and, following the considerations in Lemma \ref{lemtrbor}, we can easily show that the number of triangles that contain $v$ is $O(n^{0.2})$. On the other hand, if we remove the vertex $v$ from the graph and add a vertex to a Lenz configuration formed by $V_1$, then, from the behavior of the function $F_3(n)$ we can see that the number of triangles formed by the points of $V_1$ will increase by $\Omega (n)$, and the total number of triangles in $G$ will surely increase, if $n$ is large enough.

Thus ``moving'' all vertices from $V(1)$ to $V_1$ will increase the total number of triangles. At the end we obtain that the vertices of $G$ form a Lenz configuration, which concludes the proof of this part of the theorem.

\subsection{Schur's conjecture in $\R^4$}\label{sec33}

Consider a diameter graph $G =(V,E)$.

By Theorem \ref{th3}  any two 4-cliques in $G$ either have at least two common vertices, or do not have any. We show that $V$ can be decomposed into disjoint sets of vertices $V_1,\ldots, V_k$ with the following properties. First, any 4-clique lies entirely in one of the sets $V_1,\ldots, V_k$. Second, inside any of $V_i$ any pair of 4-cliques intersect in at least two vertices. In other words, we want to split the set of all 4-cliques into equivalence classes, in which we consider cliques equivalent if they intersect. Next, we put $V_i$ to be the union of all vertices of the cliques from the $i$-th equivalence class.

To prove that such a partition exists we need to show that this is indeed an equivalence relation. All we need to check is transitivity, i.e. that there is no such triple of 4-cliques  $K^1,K^2,K^3$, such that $|K^1\cap K^2|=|K^1\cap K^3| =2, |K^2\cap K^3| = 0.$ Note that if the cardinality of the intersection of $K^1$ with one of the rest is greater than 2, then, by the pigeonhole principle, the other two also have to intersect.

Denote by  $v_1,v_2,v_3,v_4$ the vertices of $K^1$, where $v_1,v_2\in K^2, v_3,v_4\in K^3$. The other vertices are $w_1,w_2 \in K^2,w_3,w_4\in K^3$. The hyperplane that passes through $v_1,v_2,v_3,v_4$ we denote by $\pi$. The points $v_1,v_2,w_3,w_4$ lie on a two-dimensional sphere $S_1$ of radius $\sqrt 3/2.$ Its center is the middle of the segment that connects $v_3,v_4$, while the sphere itself lies in the hyperplane $\gamma$ that is orthogonal to the segment. Analogously, the points $v_3,v_4,w_1,w_2$ lie on the two-dimensional sphere $S_2$ of radius $\sqrt 3/2,$ whose center is the midpoint of the segment $v_1v_2$.

According to Lemma \ref{lem3}, the arcs $\overset{\frown}{v_3v_4}$ and $\overset{\frown}{w_1w_2}$ (as well as $\overset{\frown}{v_1v_2}$ and $\overset{\frown}{w_3w_4}$) intersect, which implies that $w_1$ and $w_2$ (as well as $w_3$ and $w_4$) lie in the different closed halfspaces bounded by $\pi$. Indeed, $\pi\cap S_1$ is the great circle that passes through $v_1,v_2$, and $w_3,w_4$ have to be on the opposite sides of this great circle. Moreover, it is easy to derive from the proof of Lemma \ref{lem3} that none of the $w_i$ lie in the plane $\pi$. Otherwise it would be either an interior point of the arc $v_1v_2$ (or $v_3v_4$), or it would coincide with one of the $v_j$. In the first case, based on Lemma \ref{lem3}, we would obtain a contradiction with the fact that $\|v_i-w_j\|\le 1$, while $\|w_1-w_2\| = \|w_3-w_4\| = 1.$ In the second case the intersection of some two of the cliques $K_i$ would be greater than 2.

Denote by $\pi^+,\pi^-$ two open halfspaces bounded by $\pi$. W.l.o.g., $w_1,w_3\in \pi^+,$  $w_2,w_4\in\pi^-$. Consider three-dimensional spheres $S^w_1,S^w_2$ of unit radius with centers in $w_1,w_2$. They intersect with $S_1$ in the points $v_1,v_2$, and none of the two spheres $S^w_i$ contain $S_1$. Otherwise the distance from $w_1$ (or $w_2$) to any point of $S_1$ would be the same, and, by the law of cosines, the vector that connects the center $o$ of $S_1$ with $w_1$ ($w_2$) would be orthogonal to $\gamma$, which is not true. Indeed, since $w_1,w_2$ do not lie in $\pi$ but $o$ lies in $\pi,$ both $\overline{ow_1}$ and $\overline{ow_2}$ have a non-zero component that is orthogonal to $\pi$. On the other hand, since $v_3,v_4\in \pi$, the vector $\bar{u}$ that is orthogonal to $\pi$ is also orthogonal to $v_3v_4$ and, consequently, lies in the hyperplane $\gamma$. As we already established, the scalar product of $\overline{ow_1}$ ($\overline{ow_2}$) and $\bar{u}$ is nonzero, which means that $\overline{ow_1}$ and $\overline{ow_2}$ are not orthogonal to $\gamma$. Therefore, the intersections of $S_1$ and $S^w_i$ are \textit{circles} $S'_i$ on $S_1$ which pass through the points $v_1,v_2$ (see Fig. 1).


Our goal is to show that there is no room for the points $w_3,w_4$ such that all the conditions based on the fact that $G$ is a diameter graph are satisfied. The points $w_3,w_4$ lie on the sphere $S_1$. At the same time $w_3,w_4\in B^w_1\cap B^w_2\cap B^v_1\cap B_2^v,$ where $B_i^w$ are unit balls with centers at $w_i$, while $B^v_i$ are unit balls of unit radius with centers at $v_i$. By $S^v_i$ we denote the boundary sphere of $B_v^i$. We prove that the intersection $S_1\cap B^w_1\cap B^w_2\cap B^v_1\cap B_2^v$ cannot contain a pair of points at unit distance apart except for $v_1,v_2$.

Henceforth, all the considerations are limited to the hyperplane $\gamma$, and, not willing to introduce excessive notations, we modify all the notations of balls, spheres, hyperplanes and halfspaces so that the notations now correspond to these objects intersected with $\gamma$ (instead of the objects in $\R^4$). In particular, we will denote by $S^w_1,S^w_2$ two-dimensional spheres (with centers in $w_1',w_2'$), which are the intersections of the three-dimensional spheres $S^w_1,S^w_2$ with $\gamma$; by $\pi$ we denote the plane $\pi\cap \gamma$. Note that $w'_1$ lies in $\pi^{+}$ and $w'_2$ lies in  $\pi^-$.

Let $\pi_1$ be the two-dimensional plane which is orthogonal to the segment $v_1v_2$ and passes through the midpoint of the segment. The center of $S_1$ and the points $w'_1,w'_2$ all lie in $\pi_1$. We denote by  $\pi_1^+$ the open halfspace containing $v_1$, and by $\pi_1^-$ the open halfspace containing  $v_2$. Let $\pi_2$ be the two-dimensional plane which is orthogonal to both $\pi$ and $\pi_1$ and passes through the center of $S_1$. It is not difficult to see that the set $S_1\cap B^v_1\cap B_2^v$ lies entirely in the open halfspace $\pi_2^+$ that is bounded by $\pi_2$ and contains $v_1,v_2$. For $u'\in S_1$ denote by $H_{u'}$ an open hemisphere with center in $u'$. We intend to show that $S_1\cap B^v_1\cap B_2^v\subset H_u=S_1\cap \pi_2^+$, where $u$ is the midpoint of the arc $\overset{\frown}{v_1v_2}$. Since the radius of $S_1$ is greater than $1/\sqrt 2,$ we have $S_1\cap B_i^v\subset H_{v_i}$ for $i=1,2$. On the other hand, since $u\in \overset{\frown}{v_1v_2},$ we surely have $H_{v_1}\cap H_{v_2}\subset H_u$. Therefore, we have the following chain of inclusions: $S_1\cap B^v_1\cap B_2^v\subset H_{v_1}\cap H_{v_2}\subset H_u=S_1\cap \pi_2^+.$

\begin{center}  \includegraphics[width=120mm]{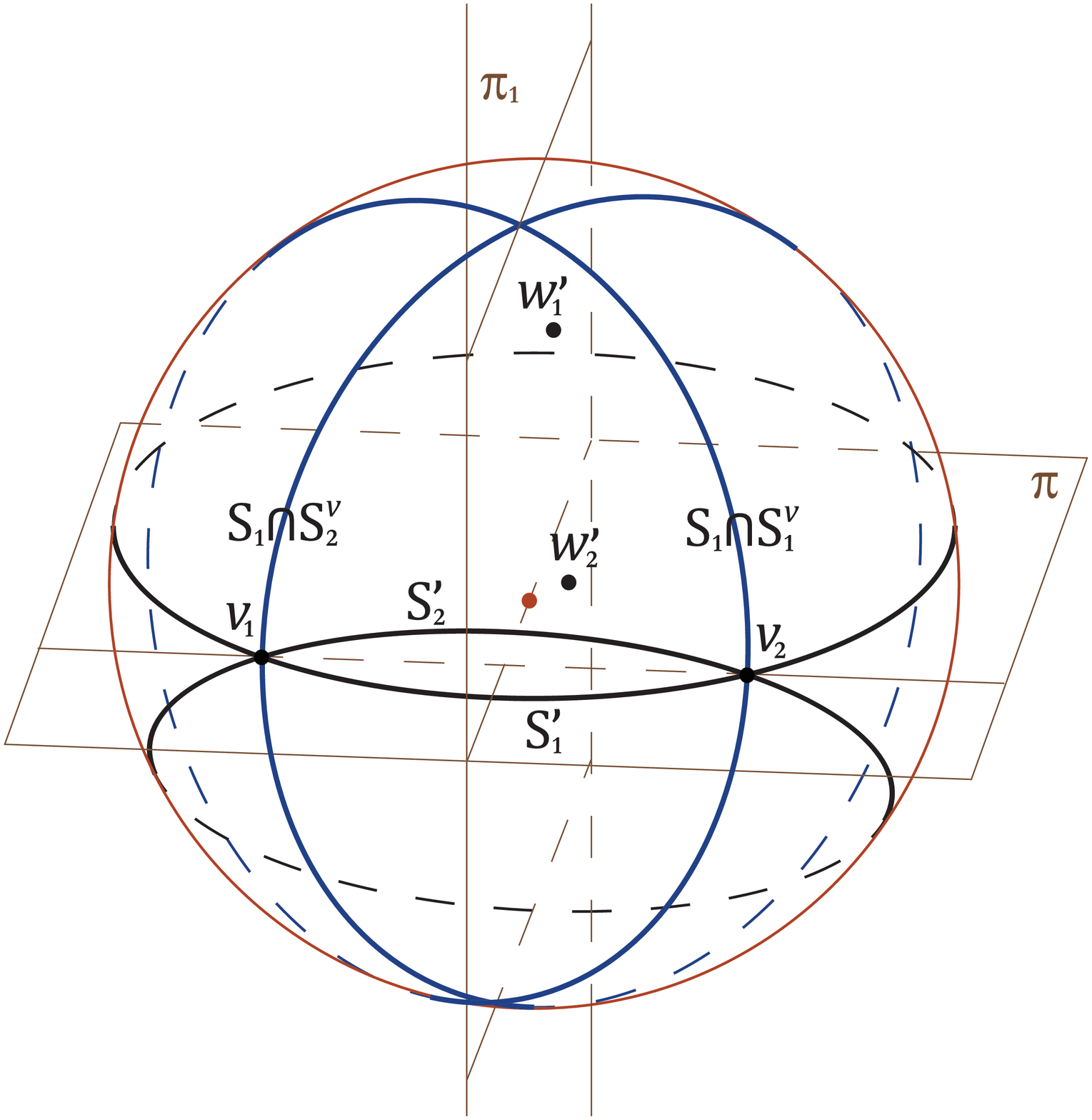}\label{pic3}  \end{center}\begin{center}  Fig. 1 \end{center}

    Next we prove that in the halfspace $\pi_1^-$  the circles $S'_1$ and $S_1\cap S^v_1$ intersect only in $v_2$. Surely, there are at most two intersection points in total. Due to the fact that $S'_1$ lies on $S_1$, the intersection of these two circles coincides with the intersection of the sphere $S_1^v$ and the circle $S'_1$. Further, since the center of $S^v_1$ lies on the circle $S'_1$, the intersection points of these two spheres should be symmetric  in the plane that contains $S'_1$ with respect to the line that contains their centers. But since one center lies in $\pi_1^+$ and the other lies in $\pi_1$,  one of the intersection points must lie in $\pi_1^+$, and $v_2$ is indeed their only intersection point in $\pi_1^-$. An analogous fact holds for the circles $S'_2$ and $S_1\cap S^v_1$, and also in the symmetric halfspace $\pi_1^+$ for the circles $S_1\cap S^v_2$ and $S'_i$ and a point $v_1$.

Recall that $w'_1\in \pi^+, w'_2 \in \pi^-$. The set $S_1\cap B^w_1$ is situated above the plane containing the circle $S'_1$ (in the direction of the normal vector to the plane $\pi$ that points to $\pi^+$). Analogously, the set $S_1\cap B^w_2$  is situated below the plane containing the circle $S'_2$. We show this for $S_1\cap B^w_1$. For this consider a reflection $\mathcal{R}:S_1\to S_1$ with respect to the plane $\pi$. Then for any point $u\in S_1\cap \pi^+$ we have $\|w'_1-u\|<\|w'_1-\mathcal R(u)\|,$ because $w'_1\in \pi^+$. The circle $S_1'$ bounds the set $S_1\cap B^w_1$, and the point on $S_1$ that is above the center of $S_1\cap B^w_1$ is closer to $w'_1$ than the point that is below the center. Note that the planes of the circles $S'_i$ cannot be orthogonal to $\pi$, since otherwise the point $w'_i$ would lie in the plane $\pi$.

The circles $S_1',S_2'$ split the sphere $S_1$ into four parts, and one of them is the set $S_1\cap B^w_1\cap B^w_2$.
From the above considerations we get that, depending on the positions of the points $w'_1,w'_2,$ the set $S_1\cap B^w_1\cap B^w_2$ has two possible locations out of four. The reason is that it is impossible that the set $S_1\cap B^w_1\cap B^w_2$
is bounded by a shorter arc $v_1v_2$ of $S_1'$ and a greater arc $v_1v_2$ of $S_2'$ (or vice versa), because in this case $v_1v_2$ of $S_1'$ is either below or above both circles $S_1', S_2'$. To prove this, we first note that from a parity argument follows that if move along the sphere $S_1$ and cross one of the circles (not in $v_1,v_2$), then, if we were in an admissible region, we arrive to a not admissible region, and vice versa. Thus, it suffices to show that the region between two shorter arcs is admissible. We already know that $v_1,v_2\in \pi_2^+$. Consider the plane $\pi'$, which is parallel to $\pi_2$  and passes through $v_1,v_2$. Any circle on $S_1$ that contains $v_1,v_2$ must have its shorter arc $v_1v_2$ in the halfspace with respect to $\pi'$ in which the point $u$ lies, which shows that the region between the two shorter arcs $v_1v_2$ of $S_1', S_2'$ is above one of the two circles and below the other.
We are left with the following two cases.

\textbf{Case 1:\ } The set $S_1\cap B^w_1\cap B^w_2$ on $S_1$ is bounded by the greater arcs of the circles $S'_i$ with the endpoints $v_1,v_2$. We specify the set $S_1\cap B^v_1\cap B^v_2$ in the following way:
$$S_1\cap B^v_1\cap B^v_2 = (S_1\cap B^v_1\cap \pi_1^-) \bigcup (S_1\cap B^v_2\cap \pi_1^+).$$

 Several paragraphs before we proved that the sets $S_1\cap B^w_1\cap B^w_2$ and $S_1\cap B^v_2\cap \pi_1^+$ intersect only in the vertex $v_1,$ while the sets $S_1\cap B^w_1\cap B^w_2$ and $S_1\cap B^v_1\cap \pi_1^-$ intersect only in the vertex $v_2.$ Thus, we obtain that

$$\bigl(S_1\cap B^v_1\cap B^v_2\bigr)\bigcap \bigl(S_1\cap B^w_1\cap B^w_2\bigr) = \{v_1,v_2\},$$
and there is no room for $w_3,w_4$ at all.

\textbf{Case 2:\ } The set $S_1\cap B^w_1\cap B^w_2$ on the sphere is bounded by the shorter arcs of the circles $S'_i$ with the endpoints in $v_1,v_2$. In that case the set $S_1\cap B^w_1\cap B^w_2$ lies entirely in the spherical cap $H$, which is cut off by the plane $\pi'$, which is parallel to $\pi_2$  and passes through $v_1,v_2$. Moreover, only the points $v_1,v_2$ lie in the intersection of $S_1\cap B^w_1\cap B^w_2$ and $\pi'$. Indeed, the set $S_1\cap B_1^w$  does not intersect with $\pi'\cap S_1\cap \pi^-$ due to the description of the position of the set $S_1\cap B^w_1$ given  before the case 1 (recall that the halfspace $\pi^+,\pi^-$ are open). Analogously, $S_1\cap B_2^w$  does not intersect with $\pi'\cap S_1\cap \pi^+$.

On the other hand, the shorter arcs of the circles $S'_i$ must lie inside $H$.

The circle $S_1\cap \pi'$ has diameter 1, and the points that lie on the sphere $S_1$ in the interior of $H$, cannot be at unit distance apart. Thus, the distance between a pair of points in $S_1\cap B^w_1\cap B^w_2$ cannot be equal to one, if these points do not coincide with $v_1,v_2$. It means that inside $S_1\cap B^w_1\cap B^w_2$ there is no room for the points $w_3,w_4$.\\

We proved that the above described partition of the vertex set $V$ into sets $V_1,\ldots V_k$ indeed exists. We apply Theorem \ref{thPM} to each $V_i$ and obtain that the number of 4-cliques on each set $V_i$ does not exceed $|V_i|$, thus the total number of cliques does not exceed $\sum_i |V_i| = n$. The proof of Schur's conjecture in $\mathbb R^4$ is complete.

\section{Acknowledgements}
The author is grateful to Alexandr Polyanskiy, who read the first version of the manuscript and made several valuable comments, and to the anonymous referee for numerous comments that helped to improve both the style and the presentation of the paper.

\end{document}